\documentclass[12pt,reqno]{amsart}

\usepackage{amssymb}
\usepackage{amsmath}
\usepackage{amscd}
\usepackage{amstext}
\usepackage{amsfonts}
\usepackage[all]{xy}

\newtheorem{thm}{Theorem}[section] 
\newtheorem{lemma}[thm]{Lemma}
\newtheorem{prop}[thm]{Proposition}
\newtheorem{cor}[thm]{Corollary}

\theoremstyle{definition}
\newtheorem{definition}[thm]{Definition}

\newcommand{\C}{\mathbb{C}}
\newcommand{\R}{\mathbb{R}}

\numberwithin{equation}{section}

\begin{document} 

\title[Higgs bundles and flat connections on Sasakian manifolds]{Higgs bundles and 
flat connections over compact Sasakian manifolds}

\author[I. Biswas]{Indranil Biswas}

\address{School of Mathematics, Tata Institute of Fundamental
Research, Homi Bhabha Road, Mumbai 400005, India}

\email{indranil@math.tifr.res.in}

\author[H. Kasuya]{Hisashi Kasuya}

\address{Department of Mathematics, Graduate School of Science, Osaka University, Osaka,
Japan}

\email{kasuya@math.sci.osaka-u.ac.jp}



\begin{abstract}
Given a compact K\"ahler manifold $X$, there is an equivalence of categories between the
completely reducible flat vector bundles on $X$ and the polystable Higgs bundles $(E,\,\theta)$
on $X$ with $c_1(E)\,=\, 0\,=\, c_2(E)$ \cite{SimC}, \cite{Cor}, \cite{UY}, \cite{DonI}. We
extend this equivalence of categories to the context of compact Sasakian manifolds. We prove that
on a compact Sasakian manifold, there is an equivalence between the category of semi-simple flat
vector bundles on it and the category of polystable basic Higgs bundles on it with trivial first
and second basic Chern classes. We also prove that any stable basic Higgs bundle over a compact
Sasakian manifold admits a basic Hermitian metric that satisfies the Yang--Mills--Higgs equation.
\end{abstract}

\maketitle

\tableofcontents

\section{Introduction}

Let $X$ be a compact K\"ahler manifold equipped with a K\"ahler form $\omega$, and let
$E$ be a holomorphic vector bundle over $X$. Uhlenbeck and Yau proved that $E$ is polystable
with respect to $\omega$ if and only if $E$ admits a Hermitian metric that satisfies
the Hermite--Einstein equation defined using $\omega$ \cite{UY}. This result was proved
earlier by Donaldson under the extra assumptions that $X$ is a complex projective
manifold and $\omega$ represents a rational cohomology class \cite{DonI}. As a consequence of
these theorems of Uhlenbeck--Yau and Donaldson, a holomorphic vector
bundle on $X$ admits a flat unitary connection if and only if $E$ is polystable with
$c_1(E)\,=\, 0\,=\, c_2(E)$.

Hitchin proved that a Higgs bundle $(E,\, \theta)$ of rank two and degree zero
on a compact Riemann surface $X$ admits a Hermitian metric that solves the Yang-Mills equation if
and only if $(E,\, \theta)$ is polystable \cite{Hi}. This implies that a polystable Higgs
bundle $(E,\, \theta)$ of rank two and degree zero produces a completely reducible flat principal
$\text{GL}(2,{\mathbb C})$--bundle (also called a principal $\text{GL}(2,{\mathbb C})$--bundle
with a semi-simple flat connection) on $X$. Donaldson proved that a principal $\text{GL}(2,
{\mathbb C})$--bundle on $X$ with a completely reducible
flat connection admits a harmonic metric \cite{Do3}; this produces the
inverse map from the completely reducible flat $\text{GL}(2,{\mathbb C})$--bundles on $X$ to the
polystable Higgs bundles of rank two and degree zero on $X$.

Simpson and Corlette extended the above bijective correspondence to the more general context
of bundles over compact K\"ahler manifolds \cite{SimC}, \cite{Cor}. More precisely, Simpson and Corlette
proved that for any compact K\"ahler manifold $X$, there is an equivalence of categories between
the completely reducible (also called semi-simple) flat vector bundles on $X$ and the polystable Higgs bundles $(E,\,
\theta)$ on $X$ with $c_1(E)\,=\, 0\,=\, c_2(E)$; this result is explained in detail in \cite{SimL}.

We recall that the contact manifolds constitute the odd dimensional counterpart of the 
symplectic manifolds. For example, while the total space of the cotangent bundle of a smooth 
manifold $X$ is a typical local model of a symplectic manifold, the local model of a contact 
manifold is the total space of the projective bundle $P(T^*M)$. In a similar vein, compact 
Sasakian manifolds can be thought of as the odd dimensional counterparts of the compact 
K\"ahler manifolds. A compact regular Sasakian manifold is in fact the unit circle bundle 
inside a holomorphic Hermitian line bundle of positive curvature on a complex projective 
manifold. More generally, a compact quasi-regular Sasakian manifold is the unit circle bundle 
inside a holomorphic Hermitian line bundle of positive curvature over a complex projective 
orbifold. However, the global structures of more general compact Sasakian manifolds, known as 
irregular Sasakian manifolds, do not admit any such explicit description.

It may be mentioned that Sasakian manifolds were introduced by Sasaki \cite{Sa}, \cite{SH}, 
which explains the terminology. During the last twenty years there has been a very substantial 
increase of the interests in Sasakian manifolds, accompanied by a flurry of research activities 
(see \cite{BoG} and references therein). It is evident from these references that a large part 
of this recent investigations into Sasakian manifolds were carried out by C. Boyer and K. 
Galicki. Another aspect contributing to this recent activities in Sasakian manifolds is the 
discovery of their relevance in the string theory in theoretical physics. This was initiated in 
the works of J. Maldacena \cite{Ma}. For further developments in this direction, see 
\cite{GMSW}, \cite{MSY}, \cite{MS} and references therein.

Our aim here is to extend, to the context of Sasakian manifolds, the earlier mentioned 
equivalence of categories between the completely reducible flat vector bundles on a compact 
K\"ahler manifold and the polystable Higgs bundles $(E,\, \theta)$ with $c_1(E)\,=\, 0\,=\, 
c_2(E)$.

The main result proved here states as follows (see Theorem \ref{thml}).

\begin{thm}\label{thm0}
For a compact Sasakian manifold $M$, there is an equivalence between the category of semi-simple
flat bundles over $M$ and the category of polystable basic Higgs bundles over $M$
with trivial first and second basic Chern classes.
\end{thm}

In \cite{BM}, the same question was addressed for the special case of quasi-regular Sasakian
manifolds. As mentioned before, quasi-regular Sasakian manifolds are the unit circle bundles
of holomorphic Hermitian line bundles of positive curvature over any complex projective
orbifold. Using this fact together with a result of geometric 
group theory \cite[p.~3492, Lemma 2.1]{BM}, the question for quasi-regular Sasakian manifolds
actually reduces to that for complex projective manifolds. In view of the known results on Higgs
bundles over complex projective manifolds, the 
differential geometric and analytical investigations needed for the proof of
Theorem \ref{thm0} could entirely be avoided in \cite{BM}.

To prove Theorem \ref{thm0} we
establish analogues of the theory of harmonic metrics on flat bundles and 
the theory of Hermitian-Yang-Mills metrics on Higgs bundles on compact K\"ahler manifolds. Just as 
Corlette and Simpson proved for compact K\"ahler manifolds (\cite[Theorem 5.1]{Cor}, 
\cite[Lemma 1.1]{SimL}), we obtain the characterization of harmonic flat bundles over compact 
Sasakian manifolds in terms of transversally holomorphic geometry (see Theorem \ref{thm1}). It
may be mentioned that this
is facilitated by a special feature of Sasakian geometry (see Theorem \ref{harmsa}) which would
not hold for a general transversally K\"ahler Geometry. 
This produces the functor that we are seeking in Theorem \ref{thm0} from the semi-simple flat bundles
to the basic Higgs bundles.

The proof of the opposite direction in Theorem \ref{thm0} is inspired by a recent work of
D. Baraglia and P. Hekmati \cite{BH}. Defining the stable and polystable basic Higgs bundles over
compact Sasakian manifolds, we prove for Sasakian manifolds the following analog of
\cite[Theorem 1]{SimC} proved in the K\"ahler setting (see Theorem \ref{HigHe}):

\begin{thm}\label{HigHe2}
For a stable basic Higgs bundle $(E, \,\theta)$ over a compact Sasakian manifold $(M,\, (T^{1,0},
\, S, \,I),\, (\eta, \,\xi))$, there exists a basic Hermitian metric $h$ on $E$ such that
\[\Lambda R^{D^{h} \perp}\, =\,0\, ,
\]
where $R^{D^{h} \perp}$ is the trace-free part of the curvature $R^{D^{h}}$ of the
canonical connection $D^{h}$ associated to $h$.
\end{thm}

Theorem \ref{HigHe2} implies the following Bogomolov--Miyaoka type inequality (see Corollary 
\ref{cor1}):

\begin{cor}
Let $(E, \,\theta)$ be a polystable basic Higgs bundle of rank $r$ over a compact
Sasakian manifold $(M,\, (T^{1,0},
\, S, \,I),\, (\eta, \,\xi))$ with $\dim M\,=\, 2n+1$. Then
$$
\int_{M}\left(2c_{2, B_{\mathcal F_{\xi}}}(E) -
\frac{r-1}{r}c_{1, B_{\mathcal F_{\xi}}}(E)^2\right)(d\eta)^{n-2}\wedge\eta \, \geq\, 0\, ,
$$
where $c_{i, B_{\mathcal F_{\xi}}}(E)$ is the $i$-th basic Chern class of $E$.
If the above inequality is an equality, 
then $R^{D^{h} \perp}\, =\,0$.
\end{cor}

We mention that the fundamental groups of compact Sasakian manifolds were investigated in
\cite{Ch}, \cite{Ka}, \cite{BM2}.

\section{Strongly pseudo-convex CR manifolds and Sasakian manifolds}

Let $M$ be a $(2n+1)$-dimensional real smooth manifold. A {\em CR-structure} on $M$ is an 
$n$-dimensional complex sub-bundle $T^{1,0}$ of the complexified tangent bundle $TM_{\C}\,=\, 
TM\otimes_{\mathbb R} {\C}$ such that $T^{1,0}\cap \overline{T^{1,0}}=\{0\}$ and $T^{1,0}$ is 
integrable (i.e., the locally defined
sections of $T^{1,0}$ are closed under the Lie bracket operation). We 
shall denote $\overline{T^{1,0}}$ by $T^{0,1}$. For a CR-structure $T^{1,0}$ on $M$, there is 
a unique sub-bundle $S$ of rank $2n$ of the real tangent bundle $TM$ together with a
vector bundle homomorphism $I\,:\,S\,\longrightarrow\, S$ satisfying the conditions that
\begin{enumerate}
\item $I^{2}\,=\,-{\rm Id}_{S}$, and

\item $T^{1,0}$ is the $\sqrt{-1}$--eigenbundle of $I$.
\end{enumerate}

A $(2n+1)$-dimensional manifold $M$ equipped with a triple $(T^{1,0},\, S,\, I)$ as above is 
called a {\em CR-manifold}. A {\em contact CR-manifold} is a CR-manifold $M$ with a contact 
$1$-form $\eta$ on $M$ such that $\ker\eta\,=\,S$. Let $\xi$ denote the Reeb vector field for the 
contact form $\eta$. On a contact CR-manifold, the above homomorphism $I$ extends to entire 
$TM$ by setting $I(\xi)\,=\,0$.

\begin{definition}
A contact CR-manifold $(M,\, (T^{1,0},\, S,\, I),\, (\eta,\, \xi))$ is a {\em 
strongly pseudo-convex CR-manifold} if the Hermitian form $L_{\eta}$ on $S_x$ defined by 
$L_{\eta}(X,Y)\,=\,d\eta(X, IY)$, $X,\,Y\,\in\, S_{x}$, is positive definite for every point
$x\,\in\, M$. 
\end{definition}

Given any strongly pseudo-convex CR-manifold $(M, \,(T^{1,0},\, S, \,I),\, (\eta,\, \xi))$, there
is a canonical Riemann metric $g_{\eta}$ on $M$, called the {\em Webster metric}, which is defined to be
\[g_{\eta}(X,Y)\,:=\, L_{\eta}(X,Y)+\eta(X)\eta(Y)\, ,\ \ X,\,Y\,\in \,T_{x}M\, , \ x\, \in\, M\, .
\] 

\begin{prop}[{\cite{Tan}, \cite{Web}}]\label{TWco}
For a strongly pseudo-convex CR-manifold $$(M,\, (T^{1,0},\, S,\, I),\, (\eta, \,\xi))$$
there exists a unique affine connection $\nabla^{TW}$ on $TM$ such that the following hold:
\begin{enumerate}
\item $\nabla^{TW}(C^{\infty}(S))\,\subset\, A^{1}(M,\,S)$, where $A^{k}(M,\,S)$ is the space of 
differential $k$-forms on $M$ with values in the vector bundle $S$.

\item $\nabla^{TW}\xi\,=\,0$, $\nabla^{TW}I\,=\,0$, $\nabla^{TW}d\eta\,=\,0$,
$\nabla^{TW}\eta\,=\,0$ and $\nabla^{TW}g_{\eta}\,=\,0$.

\item The torsion $T^{TW}$ of the affine connection $\nabla^{TW}$ satisfies the equation
\[T^{TW} (X,Y)\,=\, -d\eta (X,Y)\xi
\]
for all $X,\,Y\,\in\, S_{x}$ and $x\,\in\, M$.
\end{enumerate}
\end{prop}

The affine connection $\nabla^{TW}$ in Proposition \ref{TWco} is
called the {\em Tanaka--Webster connection}.

\begin{definition}
A Sasakian manifold is a strongly pseudo-convex CR-manifold $$(M, \,(T^{1,0},\, S,\, I),\,
(\eta,\, \xi))$$ satisfying any (all) of the following equivalent conditions:
\begin{enumerate}
\item $[\xi,\, A^{0}(M,\,S)]\,\subset\, A^{0}(M,\,S)$.

\item ${\mathcal L}_{\xi}\eta\,=\,0$ and ${\mathcal L}_{\xi}I\,=\,0$.

\item $T^{TW}(\xi,v)\,=\,0$ for all $v\, \in\, TM$.
\end{enumerate}
\end{definition}

For a Sasakian manifold, the curvature $R^{TW}$ of the Tanaka--Webster
connection $\nabla^{TW}$ satisfies the equation
\begin{equation}\label{eR}
R^{TW}(\xi,v)\,=\,0
\end{equation}
for all $v\, \in\, TM$. See \cite{BoG} for Sasakian manifolds.

\section{Harmonic metrics on Sasakian manifolds}

Let $M$ be a compact Riemannian manifold and $E$ a flat complex vector bundle over $M$
equipped with a flat connection $\nabla^{E}$.
For any Hermitian metric $h$ on $E$, we have a unique decomposition
\begin{equation}\label{code}
\nabla^{E}\,=\, \nabla^{h}+\phi_{h}
\end{equation}
such that $\nabla^{h}$ is a unitary connection and $\phi_{h}$ is a $1$-form on $M$ with values in
the self-adjoint endomorphisms of $E$ with respect to $h$.

\begin{thm}[\cite{Cor}]
If a flat complex vector bundle $(E,\, \nabla^{E})$ is semi-simple (meaning, direct sum of
irreducible flat connections), then there exists a Hermitian metric
(called the harmonic metric) $h$ on $E$ such that
\[
(\nabla^{h})^{\ast}\phi_{h}\,=\,0\, ,
\]
where $(\nabla^{h})^{\ast}$ is the formal adjoint operator of $\nabla^{h}$.
If the connection $\nabla^{E}$ is irreducible, then the harmonic metric is unique up to
multiplication by a constant scalar.
\end{thm}
 
\begin{thm}\label{harmsa}
Let $M$ be a compact Sasakian manifold with a Reeb vector field $\xi$, and let $(E,\, \nabla^{E})$
be a semi-simple flat complex vector bundle over $M$.
Then,
\[\phi_{h}(\xi)\,=\,0
\]
for any harmonic metric $h$ on $E$ for the flat connection $\nabla^{E}$.
\end{thm}

\begin{proof}
This theorem is proved by modifying the proof of \cite[Theorem 4.1]{Pet}.
Consider the vector bundle $\bigwedge T^\ast M_{\C}\otimes {\rm End}(E)$ over $M$ equipped with the
connection
$$\widetilde{\nabla}\,=\, \nabla^{TW}\otimes {\rm Id}_{{\rm End}(E)}+
{\rm Id}_{T^\ast M_{\C}}\otimes \nabla^{h}\, .
$$
Denote by $Cl(M)$ the Clifford bundle of $M$ associated with the Sasakian metric $g_{\eta}$.
Then, using the canonical isomorphism $Cl(M)\,\cong \,\bigwedge T^\ast M_{\C}$, consider
$\bigwedge T^\ast M_{\C}\otimes {\rm End}(E)$ as a Dirac bundle.
Define the Dirac operator corresponding to $\widetilde\nabla$ to be
\[{\mathcal D}\,=\,\sum_{i} e_{i}\cdot \widetilde\nabla_{e_{i}}\, ,
\]
where $\{e_{i}\}$ is a local orthonormal frame for $TM$ and ``$\cdot$'' denotes the Clifford
multiplication. Then this ${\mathcal D}$ is in fact a formal self-adjoint operator
(see \cite[Proposition 2.1]{Pet}).

We have the following formula (cf. \cite[Lemma 3.1]{Pet}). 

\begin{lemma}
${\mathcal D}\phi_{h}\,=\,-d\eta\otimes \phi_{h}(\xi)$. 
\end{lemma}

\begin{proof}
By the flatness of the connection $\nabla^{E}$, we have $(\nabla^{h}+\phi_{h})^{2}\,=\,0$, and
this implies that
\begin{equation}\label{Rh}
R^{h}\,=\,-\frac{1}{2}[\phi_{h},\,\phi_{h}]
\end{equation}
and $\nabla^{h}\phi_{h}\,=\,0$, where
$R^{h}$ is the curvature of the
Hermitian connection $\nabla^{h}$ on $E$ associated to $h$ (see \eqref{code}).
By a computation as in the proof of \cite[Lemma 3.1]{Pet}, we conclude that
${\mathcal D}\phi_{h}-(\nabla^{h})^{\ast}\phi_{h}$ is the anti-symmetrization of
covariant derivative on $\bigwedge T^{\ast}M_{\C}\otimes {\rm End}(E)$, and 
we have $$({\mathcal D}\phi_{h}-(\nabla^{h})^{\ast}\phi_{h})(X,Y)
\,=\,(\nabla^{h}\phi_{h})(X,Y) -\phi_{h}(T^{TW}(X,Y))\, .$$
By the harmonicity of $h$, $$(\nabla^{h})^{\ast}\phi_{h}\,=\,0\, =\, \nabla^{h}\phi_{h}$$
as above. Now the lemma follows from Proposition \ref{TWco}.
\end{proof}

Consider the following formula
\[\frac{1}{2}[{\mathcal D}^{2},\, (\xi\cdot)]\,=\, \frac{1}{2}({\mathcal D}^{2}(\xi\cdot)-
\xi\cdot {\mathcal D}^{2})\,=\,-2\xi\cdot {\mathcal R}_{\xi}
\]
(see \cite[Corollary 2.1]{Pet}),
where ${\mathcal R}_{\xi}$ is the endomorphism defined by
\[{\mathcal R}_{\xi}\,=\,-\frac{1}{2}\sum_{i} \xi\cdot e_{i}\cdot \widetilde{R}(\xi,e_{i})
\]
with $\widetilde{R}$ being the curvature of the connection $\widetilde\nabla\,=\,
\nabla^{TW}\otimes {\rm Id}_{{\rm End}(E)}+{\rm Id}_{T^\ast M_{\C}}\otimes \nabla^{h}$
on $T^\ast M_{\C}\otimes{\rm End}(E)$.
For convenience, we take $\{e_{i}\}$ to be a local orthonormal frame of $TM$ such that
$e_{0}\,=\,\xi$ and $e_{1},\,\cdots ,\,e_{2n}$ is a local orthonormal frame of $S$ associated with $L_{\eta}$. 
We have
$$
\frac{1}{2}\int_{M}\langle [{\mathcal D}^{2},(\xi\cdot)]\phi_{h},\,\xi\cdot \phi_{h}\rangle
\,=\,-2\int_{M}\langle \xi \cdot{\mathcal R}_{\xi}\phi_{h},\,\xi\cdot\phi_{h}\rangle
\,=\,-2\int_{M}\langle {\mathcal R}_{\xi}\phi_{h},\,\phi_{h}\rangle
$$
$$
=\, \sum_{i=0}^{2n}\int_{M}\langle \xi\cdot e_{i}\cdot R(\xi,e_{i})\phi_{h},\, \phi_{h}\rangle
\,=\,-\sum_{i=0}^{2n}\int_{M}\langle
e_{i}\cdot R(\xi,e_{i})\phi_{h},\, \xi\cdot\phi_{h}\rangle\, .
$$
By $R^{TW}(\xi,-)\,=\,0$ (see \eqref{eR}), we have that
$$
-\sum_{i=0}^{2n}\int_{M}\langle e_{i}\cdot R(\xi,e_{i})\phi_{h},\, \xi\cdot\phi_{h}\rangle
$$
$$
=\,-\sum_{i=1}^{2n}\int_{M}\left(\langle (R^{h}(\xi,e_{i})\phi_{h})(e_{i}),\, \phi_{h}(\xi)\rangle-\langle 
(R^{h}(\xi,e_{i})\phi_{h})(\xi),\, \phi_{h}(e_{i})\rangle\right)
$$
(see \cite[formula (17)]{Pet}).
Using $R^{h}\,=\,-\frac{1}{2}[\phi_{h},\,\phi_{h}]$ (see \eqref{Rh}), this is equal to
\[
-\frac{1}{2}\sum_{i=1}^{2n}\int_{M}\left(\langle[[\phi_{h}(\xi),\phi_{h}(e_{i})],\phi_{h}(e_{i})],
\,\phi_{h}(\xi)\rangle-\langle [[\phi_{h}(\xi),\,\phi_{h}(e_{i})],\,\phi_{h}(\xi)],\, \phi_{h}(e_{i})\rangle\right).
\]
Since $\phi_{h}$ is a $1$-form with values in the self-adjoint endomorphisms of $E$, this is equal to 
 \[-\sum_{i=1}^{2n}\int_{M}\langle[\phi_{h}(\xi),\,\phi_{h}(e_{i})],\, [\phi_{h}(\xi),\,\phi_{h}(e_{i})]\rangle.
 \]
 Thus we obtain the inequality 
\begin{equation}\label{inq1}
\frac{1}{2}\int_{M}\langle [{\mathcal D}^{2},\,(\xi\cdot)]\phi_{h},\,\xi\cdot \phi_{h}\rangle\,\le
\, 0\, .
\end{equation}
 
On the other hand, we can directly compute that
\[\frac{1}{2}\int_{M}\langle [{\mathcal D}^{2},\,(\xi\cdot)]\phi_{h},\,\xi\cdot \phi_{h}\rangle
\,=\,4\int_{M}\left(\langle\xi\cdot {\mathcal D}(\phi_{h}),\,
\widetilde\nabla_{\xi}\phi_{h}\rangle +\langle
\widetilde\nabla_{\xi}\phi_{h},\,\widetilde\nabla_{\xi}\phi_{h}\rangle\right)
\]
as done in \cite[Page 594]{Pet}.
Using $(\xi\cdot)\,=\,(\eta\wedge) -i_{\xi}$ and ${\mathcal D}\phi_{h}\,=\,-d\eta\otimes
\phi_{h}(\xi)$, where $i_{\xi}$ denotes the interior product,
together with the fact that $i_{\xi}d\eta\,=\,0$ (this means that the form $d\eta$ is basic; see
\eqref{bafo}), we
conclude that $\langle\xi\cdot {\mathcal D}(\phi_{h}),\,\widetilde\nabla_{\xi}\phi_{h}\rangle\,=\,0$.
Thus, we have the inequality 
\begin{equation}\label{inq2}
\frac{1}{2}\int_{M}\langle [{\mathcal D}^{2},(\xi\cdot)]\phi_{h},\,\xi\cdot \phi_{h}\rangle\,=\,
 \int_{M}\langle \widetilde\nabla_{\xi}\phi_{h},\,\widetilde\nabla_{\xi}\phi_{h}\rangle\,\ge \,0\, .
\end{equation}
Now from \eqref{inq1} and \eqref{inq2} it follows that $\widetilde\nabla_{\xi}\phi_{h}\,=\,0$.

By the same argument as in \cite[Page 598]{Pet}, we have that $\phi_{h}(\xi)\,=\,0$.
\end{proof}
 
\section{Basic vector bundles}

Let $M$ be a compact manifold equipped with a nonsingular foliation $\mathcal F$. Then, a 
differential form $\omega$ on $M$ is called {\em basic} if for every vector field $X$ on $M$ 
which is tangent to the leaves of $\mathcal F$, the equations
\begin{equation}\label{bafo}
i_{X}\omega\,=\,0\,=\, {\mathcal L}_{X}\omega
\end{equation}
hold.

We denote by $A^{\ast}_{B_{\mathcal F}}(M)$ the subspace of basic 
forms in the de Rham complex $A^{\ast}(M)$. It is straight-forward to check that
$A^{\ast}_{B_{\mathcal F}}(M)$ is a sub-complex of the de Rham complex $A^{\ast}(M)$.
Denote by
\begin{equation}\label{bc}
H_{B_{\mathcal F_{\xi}}}^{\ast}(M)\,=\, \bigoplus_{i\geq 0}
H_{B_{\mathcal F_{\xi}}}^{i}(M)
\end{equation}
the cohomology of the basic de Rham complex $A^{\ast}_{B_{\mathcal F_{\xi}}}(M)$. Note that
there is a natural homomorphism from $H_{B_{\mathcal F_{\xi}}}^{i}(M)$ to the $i$-th
de Rham cohomology of $M$.

Let $E$ be a complex $C^\infty$ vector bundle over $M$ of rank $r$. This $E$ is said to be 
{\em basic} if it has local trivializations with respect to an open covering 
$M\,=\,\bigcup_{\alpha} U_{\alpha}$ satisfying the condition that each transition function 
$f_{\alpha\beta}\,:\, U_{\alpha}\cap U_{\beta}\,\longrightarrow\, {\rm GL}(r, \C)$ is basic 
on $U_{\alpha}\cap U_{\beta}$, meaning it is constant on the leaves of the foliation 
$\mathcal F$. This condition is equivalent to the condition that
$E$ has a flat partial connection in the direction of $\mathcal F$.

For a basic vector bundle $E$ over $M$, a differential form $\omega\,\in\, A^{\ast}(M,\,E)$ with 
values in $E$ is called basic if $\omega$ is basic on every $U_{\alpha}$, meaning 
$\omega_{\vert U_{\alpha}}\,\in\, A^{\ast}_{B_{\mathcal F}}(U_{\alpha})\otimes \C^{r}$ for 
every $\alpha$. Let $$A^{\ast}_{B_{\mathcal F}}(M,\, E)\,\subset\,A^{\ast}(M,\,E)$$ denote the 
subspace of basic forms in the space $A^{\ast}(M,\,E)$ of differential forms with values in 
$E$.

We shall consider any flat vector bundle $(E,\, \nabla^{E})$ over $M$ as a basic vector 
bundle by local flat frames. Then, $A^{\ast}_{B}(M,\, E)$ is a sub-complex of the de Rham 
complex $A^{\ast}(M,\,E)$ equipped with the differential $d^{E}$ associated to the flat 
connection $\nabla^{E}$.

Let $(E,\, \nabla^{E})$ be a flat vector bundle over $M$.
For a Hermitian metric $h$ on $E$, consider the canonical decomposition
\[\nabla^{E}\,=\,\nabla^{h}+\phi_{h}
\]
in \eqref{code}.

\begin{prop}\label{Hba}
The following two conditions are equivalent:
\begin{itemize}
\item $\phi_{h}(X)\,=\,0$ for all $X\,\in\, T_{x}{\mathcal F}$ and $x\,\in\, M$.

\item The Hermitian structure $h$ is basic, meaning $h\,\in\, A_{B_{\mathcal F}}^{0}(M,\,
E^{\ast}\otimes \overline{E}^{\ast})$.
\end{itemize}
These conditions imply that $\phi_{h}\,\in\, A^{1}_{B_{\mathcal F}}(M,\, {\rm End}(E))$.
\end{prop}

\begin{proof}
For local flat frames of $E$ with respect to an open covering $M\,=\,\bigcup_{\alpha} U_{\alpha}$, we can 
write $$\phi_{h}\,=\,-\frac{1}{2}df_{\alpha}f_{\alpha}^{-1}$$ on each $U_{\alpha}$ for certain 
functions $f_{\alpha}$ on $U_{\alpha}$ with values in the positive definite Hermitian matrices 
$Herm^{+}_{r}$ with respect to $h$. Then, $\phi_{h}(X)\,=\,0$ for any $X\,\in\, T_{x}{\mathcal F}$ if and only if each 
$f_{\alpha}$ is basic. Thus the proposition follows.
\end{proof}
 
Let $(M,\, (T^{1,0}, \,S,\, I),\, (\eta, \,\xi))$ be a compact Sasakian manifold. Then the 
Reeb vector field $\xi$ defines a $1$-dimensional foliation $\mathcal F_{\xi}$ on $M$. It is 
known that the map $I\,:\,TM\,\longrightarrow\, TM$ associated with the CR-structure $T^{1,0}$ 
defines a transversely complex structure on the foliated manifold $(M,\,\mathcal F_{\xi})$. Furthermore, 
the closed basic $2$-form $d\eta$ is a transversely K\"ahler structure with respect
to this transversely complex structure. Corresponding to the
decomposition $S_{\C}\,=\,T^{1,0}\oplus T^{0,1}$, we have the bigrading $$A^{r}_{B_{\mathcal 
F_{\xi}}}(M)_{\C}\,=\,\bigoplus_{p+q=r} A^{p,q}(M)$$ as well as the decomposition of the exterior 
differential $$d_{\vert A^{r}_{B_{\mathcal 
F_{\xi}}}(M)_{\C}}\,=\,\partial_{\xi}+\overline\partial_{\xi}$$ on $A^{r}_{B_{\mathcal 
F_{\xi}}}(M)_{\C}$, so that $$\partial_{\xi}\,:\,A^{p,q}_{B_{\mathcal F}}(M)\,\longrightarrow\, 
A^{p+1,q}_{B_{\mathcal F}}(M)\ \text{ and } \
\overline\partial_{\xi}\,:\,A^{p,q}_{B_{\mathcal F}}(M)\,\longrightarrow\, 
A^{p,q+1}_{B_{\mathcal F}}(M)\, .$$
 
We shall now use transverse Hodge theory (\cite{KT}, \cite{EKA}).
Consider the usual Hodge star operator $$\ast\,:\, A^{r}(M)\,\longrightarrow\, A^{2n+1-r}(M)$$
associated to the Sasakian metric $g_{\eta}$ and the formal adjoint operator
$$\delta\,=\,-\ast d\ast \,:\,A^{r}(M)\,\longrightarrow\, A^{r-1}(M)\, .$$
We define the homomorphism $$\star_{\xi}\,:\, A^{r}_{B_{\mathcal F_{\xi}}}(M)
\,\longrightarrow\,A^{2n-r}_{B_{\mathcal F_{\xi}}}(M)$$ to be
$\star_{\xi}\omega\,=\,\ast(\eta\wedge \omega)$ for $\omega\,\in\, A^{r}_{B_{\mathcal F_{\xi}}}(M)$.
Also define the operators $$\delta_{\xi}\,=\,-\star_{\xi}d\star_{\xi}\,:\,
A^{r}_{B_{\mathcal F_{\xi}}}(M)\,\longrightarrow\, A^{r-1}_{B_{\mathcal F_{\xi}}}(M)\, ,$$
$$\partial_{\xi}^{\ast}\,=\,-\star_{\xi}\overline\partial_{\xi}\star_{\xi}\,:\,
A^{p,q}_{B_{\mathcal F}}(M)\,\longrightarrow\, A^{p-1,q}_{B_{\mathcal F}}(M)\, ,$$
$$\overline\partial_{\xi}^{\ast}\,=\,-\star_{\xi}\partial_{\xi}\star_{\xi}\,:\,
A^{p,q}_{B_{\mathcal F}}(M)\,\longrightarrow\, A^{p,q-1}_{B_{\mathcal F}}(M)$$ and
$\Lambda \,=\,-\star_{\xi}(d\eta\wedge)\star_{\xi}$.
They are the formal adjoints of $d$, $\partial_{\xi}$, $\overline\partial_{\xi}$ and
$(d\eta\wedge)$ respectively for the pairing 
\[A^{r}_{B_{\mathcal F_{\xi}}}(M)\times A^{r}_{B_{\mathcal F_{\xi}}}(M)\,\ni\,
(\alpha,\,\beta)\,\longmapsto\, \int_{M} \eta\wedge\alpha\wedge \star_{\xi}\beta\, .
\]
Define the Laplacian operators $$\Delta\,:\, A^{r}(M)\,\longrightarrow\, A^{r}(M)\ \ \text{ and }
\ \ \Delta_{\xi}\,:\,A^{r}_{B_{\mathcal F_{\xi}}}(M)\,\longrightarrow\,
A^{r}_{B_{\mathcal F_{\xi}}}(M)$$ by
\[\Delta\,=\,d\delta+\delta d\ \ \text{ and } \ \ \Delta_{\xi}\,=\,d\delta_{\xi}+\delta_{\xi}d
\]
respectively. For $\omega\,\in\, A^{r}_{B_{\mathcal F_{\xi}}}(M)$, since the relation $\ast\omega
\,=\,(\star_{\xi}\omega)\wedge \eta$ holds, we have the relation 
\[\delta\omega\,=\,\delta_{\xi}\omega+\ast (d\eta\wedge \star_{\xi}\omega)\, .
\]
Thus, for $\omega\,\in\, A^{1}_{B_{\mathcal F_{\xi}}}(M)$, the equality $\delta_{\xi}\omega
\,=\,\delta\omega$ holds, and hence for $f\,\in\, A^{0}_{B_{\mathcal F_{\xi}}}(M)$, we have
that $\Delta_{\xi}f\,=\,\Delta f$.
The usual K\"ahler identities
\[[\Lambda, \partial_{\xi}]\,=\,
-\sqrt{-1}\overline\partial_{\xi}^{\ast}\ \ \text{ and } \ \
[\Lambda ,\overline\partial_{\xi}]=\sqrt{-1}\partial_{\xi}^{\ast}
\]
hold, and these imply that
\[\Delta_{\xi}\,=\,2\Delta_{\xi}^{\prime}\,=\,2\Delta_{\xi}^{\prime\prime}\, ,
\]
where $\Delta_{\xi}^{\prime}\,=\,\partial_{\xi}\partial_{\xi}^{\ast}+\partial_{\xi}^{\ast}\partial_{\xi}$
and $\Delta_{\xi}^{\prime\prime}\,=\,\overline\partial_{\xi}\overline\partial_{\xi}^{\ast}+
\overline\partial_{\xi}^{\ast}\overline\partial_{\xi}$.

Let $E$ be a complex $C^\infty$ vector bundle over $M$.
We say that $E$ is transversely holomorphic if 
it admits local trivializations with respect to an open covering $M\,=\,\bigcup_{\alpha} U_{\alpha}$ such that
each transition function $f_{\alpha\beta}\,:\, U_{\alpha}\cap U_{\beta}\,\longrightarrow\,
{\rm GL}(r,\C)$ is basic and holomorphic (i.e., $\overline\partial_{\xi} f_{\alpha\beta}\,=\,0$).
For a transversely holomorphic vector bundle $E$ over $M$, define the canonical
Dolbeault operator $$\overline\partial_{E}\,:\,A^{p,q}_{B_{\mathcal F}}(M,\, E)\,\longrightarrow\,
A^{p,q+1}_{B_{\mathcal F}}(M,\, E)$$ satisfying the following two conditions:
$$\overline\partial_{E} (\omega\wedge s)\,=\,
(\overline\partial_{\xi}\omega)\wedge s+(-1)^{p+q}\omega\wedge \overline\partial_{E} s$$ and
$\overline\partial_{E}\overline\partial_{E} \,=\,0$.
Conversely, on a basic complex vector bundle $E$, if we have an operator
$\overline\partial_{E}\,:\,A^{p,q}_{B_{\mathcal F}}(M,\, E)\,\longrightarrow\,
A^{p,q+1}_{B_{\mathcal F}}(M, \,E)$ satisfying the above two conditions, then the
differential operator $\overline\partial_{E}$ defines a
canonical transversely holomorphic structure on $E$ by the Frobenius theorem \cite[p.~9,
Proposition 3.7]{Ko}.

Let $(E,\, \nabla^{E})$ be a flat complex vector bundle
over $M$ such that $E$ is equipped with a Hermitian metric $h$. Let
$$\nabla^{E}\,=\,\nabla^{h}+\phi_{h}$$ 
be the canonical decomposition of the connection $\nabla^{E}$ (see \eqref{code}).
Then, using the pairing $A^{r}(M,\,E)\times A^{2n+1-r}(M,\,E)\,\longrightarrow\, A^{2n+1}(M)$
associated with $h$, define the Hodge star operator $$\ast_{h}\,:\, A^{r}(M,\,E)
\,\longrightarrow\, A^{2n+1-r}(M,\,E)$$ as well as the formal adjoint operator
$(\nabla^{h})^{\ast}\,=\,-\ast_{h}\nabla^{h}\ast_{h}$.

Assume the Hermitian structure $h$ to be basic (equivalently, $\phi_{h}(\xi)\,=\,0$ by Proposition \ref{Hba}).
Since $\nabla^{E}_{\xi}\,=\,\nabla^{h}_{\xi}$, it follows that the unitary connection
$\nabla^h$ restricts to a homomorphism
$$\nabla^{h}\,:\, A^{r}_{B_{\mathcal F}}(M, \,E)\,\longrightarrow\,
A^{r+1}_{B_{\mathcal F}}(M,\, E)\, .$$
Now, on $A^{p,q}_{B_{\mathcal F_{\xi}}}(M,\, E)$, decompose $$\nabla^{h}\,=\,
\partial_{h,\xi}+\overline\partial_{h,\xi}$$ such that $\partial_{h,\xi}\,:\,
A^{p,q}_{B_{\mathcal F_{\xi}}}(M,\, E) \,\longrightarrow\, A^{p+1,q}_{B_{\mathcal F_{\xi}}}(M,\, E)$
and $\overline\partial_{h,\xi}\,:\,A^{p,q}_{B_{\mathcal F_{\xi}}}(M,\, E)
\,\longrightarrow\, A^{p,q+1}_{B_{\mathcal F_{\xi}}}(M,\, E)$.
Since $\nabla^{h}$ is a unitary connection, we have that $$\overline\partial_{\xi}h(s_{1},s_{2})
\,=\,h(\overline\partial_{h,\xi}s_{1}, s_{2})+h(s_{1}, \partial_{h,\xi}s_{2})$$ for
$s_1,\, s_2\,\in\, A^{0,0}_{B_{\mathcal F_{\xi}}}(M,\, E)$.
We define the operator $$\star_{h, \xi}\,:\, A^{r}_{B_{\mathcal F_{\xi}}}(M,\,E)\,\longrightarrow\,
A^{2n-r}_{B_{\mathcal F_{\xi}}}(M,\,E)$$ and the formal adjoint operators
$$(\nabla^{h})^{\ast}_{\xi}\,=\,-\star_{h, \xi}\nabla^{h}\star_{h,\xi}\,:\,
A^{r}_{B_{\mathcal F_{\xi}}}(M,\,E)\,\longrightarrow\, A^{r-1}_{B_{\mathcal F_{\xi}}}(M,\,E)\, ,$$
$\partial_{h,\xi}^{\ast}\,=\,-\star_{h,\xi}\overline\partial_{h,\xi}\star_{h,\xi}\,:\,
A^{p,q}_{B_{\mathcal F_{\xi}}}(M,\,E)\,\longrightarrow\, A^{p-1,q}_{B_{\mathcal F_{\xi}}}(M,\,E)$ and
$$\overline\partial_{h,\xi}^{\ast}\,=\,-\star_{h,\xi}\partial_{h,\xi}\star_{h,\xi}\,:\,
A^{p,q}_{B_{\mathcal F_{\xi}}}(M,\,E)\,\longrightarrow\,
A^{p,q-1}_{B_{\mathcal F_{\xi}}}(M,\,E)\, ,$$ as well we
$\Lambda_{h} \,:=\,-\star_{h,\xi}(d\eta\wedge)\star_{h,\xi}$ in the same way as above.
We now have the K\"ahler identities
\[[\Lambda,\, \partial_{h,\xi}]\,=\,-\sqrt{-1}\overline\partial_{h,\xi}^{\ast} \ \
\text{ and }\ \ [\Lambda ,\,\overline\partial_{h,\xi}]\,=\, \sqrt{-1}\partial_{h,\xi}^{\ast}\, .
\]

\begin{thm}\label{thm1}
Let $(M,\, (T^{1,0},\, S,\, I),\, (\eta,\, \xi))$ be a compact Sasakian manifold and
$(E, \, \nabla^{E})$ a flat complex vector bundle over $M$ with a Hermitian metric $h$.
Then the following two conditions are equivalent:
\begin{itemize}
\item The Hermitian structure $h$ is harmonic, i.e., $(\nabla^{h})^{\ast}\phi_{h}\,=\,0$.

\item The Hermitian structure $h$ is basic ($\Leftrightarrow \phi_{h}(\xi)\,=\,0$ and implying
that $$\phi_{h}\,\in 
\,A^{1}_{B_{\mathcal F_{\xi}}}(M, {\rm End}(E))$$ by Proposition \ref{Hba}), and for the 
decomposition $$\phi_{h}\,=\,\theta^{1,0}_{h,\xi}+\theta^{0,1}_{h,\xi}$$ with 
$\theta^{1,0}_{h,\xi}\,\in\, A^{1,0}_{B_{\mathcal F_{\xi}}}(M,\,{\rm End}(E))$ and $ 
\theta^{0,1}_{h,\xi}\in A^{0,1}_{B_{\mathcal F_{\xi}}}(M,\,{\rm End}(E))$,
\[\overline\partial_{h,\xi}\overline\partial_{h,\xi}\,=\,0,\qquad 
[\theta^{1,0}_{h,\xi},\,\theta^{1,0}_{h,\xi}]\,=\,0\qquad {\rm and} \qquad 
\overline\partial_{h,\xi}\theta^{1,0}_{h,\xi}\,=\,0\, .\]
\end{itemize}
\end{thm}
 
\begin{proof}
First suppose that $h$ is basic. 
Then by the relation similar to the relation between $\delta$ and $\delta_{\xi}$, we have that 
\[(\nabla^{h})^{\ast}\phi_{h}\,=\,(\nabla^{h})^{\ast}_{\xi}\phi_{h}\, .
\]
By the same computation as in \cite[Page 376]{Cor}, we have that
\[(\nabla^{h})^{\ast}_{\xi}\phi_{h}\,=\,\sqrt{-1}\Lambda_{h}(\partial_{h,\xi}-
\overline\partial_{h,\xi})\phi_{h}
\, =\,2\partial_{h,\xi}^{\ast}\theta^{1,0}_{h,\xi}\, .
\]
Thus, $h$ is harmonic if and only if $\overline\partial_{h,\xi}\theta^{1,0}_{h,\xi}$ is primitive,
meaning $\Lambda_{h}\overline\partial_{h,\xi}\theta^{1,0}_{h,\xi}\,=\,0$. From
this it follows that the second condition in the theorem implies the first one.
 
To prove the converse, suppose that $h$ is harmonic.
Then, by Theorem \ref{harmsa}, $h$ is basic. 
By the flatness of the connection $\nabla^E$, we have that $(\nabla^{h})^{2}\,=\,-\frac{1}{2}[\phi_{h},\phi_{h}]$
and $\nabla^{h}\phi_{h}\,=\,0$. Now using
$(\nabla^{h})^{2}\,=\,-\frac{1}{2}[\phi_{h},\phi_{h}]$ it follows that
\[\partial_{h,\xi}\partial_{h,\xi}\,=\,-[\theta^{1,0}_{h,\xi},\,\theta^{1,0}_{h,\xi}]
\qquad {\rm and} \qquad \overline\partial_{h,\xi}\overline\partial_{h,\xi}\,=
\,-[\theta^{0,1}_{h,\xi},\, \theta^{0,1}_{h,\xi}]\, .
\]
Now $\nabla^{h}\phi_{h}\,=\,0$ implies that
\[\partial_{h,\xi}\theta^{1,0}_{h,\xi},\qquad \overline\partial_{h,\xi}\theta^{0,1}_{h,\xi}\qquad {\rm and}\qquad \overline\partial_{h,\xi}\theta^{1,0}_{h,\xi}+\partial_{h,\xi}\theta^{0,1}_{h,\xi}=0.
\]
 
We have (cf. the proof of \cite[Theorem 5.1]{Cor})
$$
\partial_{\xi}\overline\partial_{\xi}h(\theta^{1,0}_{h,\xi},\, \theta^{1,0}_{h,\xi})
\,=\, h(\partial_{h,\xi}\overline\partial_{h,\xi}\theta^{1,0}_{h,\xi},\, \theta^{1,0}_{h,\xi})+h(\overline\partial_{h,\xi}\theta^{1,0}_{h,\xi},\,\overline\partial_{h,\xi}\theta^{1,0}_{h,\xi})
$$
$$
=\,-h(\partial_{h,\xi}\partial_{h,\xi}\theta^{0,1}_{h,\xi},\, \theta^{1,0}_{h,\xi})+h(\overline\partial_{h,\xi}\theta^{1,0}_{h,\xi},\,\overline\partial_{h,\xi}\theta^{1,0}_{h,\xi})
\,=\,h\left([\theta^{1,0}_{h,\xi},\,\theta^{1,0}_{h,\xi}],\,\theta^{0,1}_{h,\xi}],\,
\theta^{1,0}_{h,\xi}\right)+h\left(\overline\partial_{h,\xi}\theta^{1,0}_{h,\xi},
\,\overline\partial_{h,\xi}\theta^{1,0}_{h,\xi}\right)
$$
$$
=\, -h\left([\theta^{1,0}_{h,\xi},\,\theta^{1,0}_{h,\xi}],\,[\theta^{1,0}_{h,\xi},
\,\theta^{1,0}_{h,\xi}]\right)+h\left(\overline\partial_{h,\xi}\theta^{1,0}_{h,\xi},
\,\overline\partial_{h,\xi}\theta^{1,0}_{h,\xi}\right)\, .
$$
Since $\overline\partial_{h,\xi}\theta^{1,0}_{h,\xi}$ is primitive by the above argument,
using the Lefschetz decomposition of basic forms corresponding to the transversely K\"ahler form
$d\eta$, and integrating the wedge product of this equation and $(d\eta)^{n-2}\wedge \eta$,
from the Stokes theorem, we obtain that
\[0\,=\,-C_{1}\int_{M}\vert \overline\partial_{h,\xi}\theta^{1,0}_{h,\xi}\vert -C_{2}\int_{M}\vert [\theta^{1,0}_{h,\xi},\,\theta^{1,0}_{h,\xi}]\vert
\]
for some positive constant $C_{1}$ and $C_{2}$. Consequently, we have 
 \[ [\theta^{1,0}_{h,\xi},\,\theta^{1,0}_{h,\xi}]\,=\,0\qquad {\rm and} \qquad
\overline\partial_{h,\xi}\theta^{1,0}_{h,\xi}\,=\,0\, .\]
Using $\partial_{h,\xi}\partial_{h,\xi}\,=\,-[\theta^{1,0}_{h,\xi},\,\theta^{1,0}_{h,\xi}]$
it follows that $\overline\partial_{h,\xi}\overline\partial_{h,\xi}\,=\,0$
and hence the theorem is proved.
\end{proof}
 
\section{Basic Higgs bundles}\label{se5}

Let $(M,\, (T^{1,0},\, S,\, I),\, (\eta,\, \xi))$ be a compact Sasakian manifold.
As in \eqref{bc}, $H_{B_{\mathcal F_{\xi}}}^{\ast}(M)$ is the cohomology of the basic
de Rham complex $A^{\ast}_{B_{\mathcal F_{\xi}}}(M)$.
By the K\"ahler identities on $A^{p,q}_{B_{\mathcal F_{\xi}}}(M)$, as in the usual K\"ahler
case, we have the canonical Hodge decomposition
\[H_{B_{\mathcal F_{\xi}}}^{r}(M)\otimes \C\,=\,\bigoplus_{p+q=r} H^{p,q}_{B_{\mathcal F_{\xi}}}(M)\, .
\]

Let $E$ be a complex basic vector bundle over $M$.
Consider a connection operator $$\nabla\,:\, A^{\ast}_{B_{\mathcal F_{\xi}}}(M,\,E)
\,\longrightarrow\, A^{\ast+1}_{B_{\mathcal F_{\xi}}}(M,\,E)$$ satisfying the equation
\[\nabla(\omega s)\,=\, (d\omega )s+(-1)^{r}\omega\wedge \nabla s
\]
for $\omega\,\in\, A^{r}_{B_{\mathcal F_{\xi}}}(M)$ and $s\,\in\,
A^{0}_{B_{\mathcal F_{\xi}}}(M,\,E)$.
Let $$R^{\nabla}\,= \,\nabla^2\,\in\, A_{B_{\mathcal F}}^{2}({M,\, \rm End}(E))$$
be the curvature of $\nabla$. For
any $1\,\le\, i\,\le \,n$,
Define $c_{i, B_{\mathcal F_{\xi}}}(E,\nabla) \,\in\, A^{2i}_{B_{\mathcal F_{\xi}}}(M)$ by 
\[{\rm det}\left(I- \frac{R^{\nabla}}{2\pi\sqrt{-1}}\right)
\,=\,1+\sum_{i=1}^{n}c_{i, B_{\mathcal F_{\xi}}}(E,\nabla) \, .
\]
Then, as the case of usual Chern-Weil theory, the cohomology class
$$c_{i, B_{\mathcal F_{\xi}}}(E)\, \in\, H_{B_{\mathcal F_{\xi}}}^{2i}(M)$$ of each
$c_{i, B_{\mathcal F_{\xi}}}(E,\nabla)$ is independent of the choice of the connection $\nabla$
taking $A^{\ast}_{B_{\mathcal F_{\xi}}}(M,\,E)$ to $A^{\ast+1}_{B_{\mathcal F_{\xi}}}(M,\,E)$.
If $E$ is a transversely holomorphic vector bundle, just as in the case of Chern classes of
holomorphic vector bundles over compact K\"ahler manifolds, we have that
$c_{i, B_{\mathcal F_{\xi}}}(E)\,\in\, H^{i,i}_{B_{\mathcal F_{\xi}}}(M)$.

A \textit{basic Higgs bundle} over $(M,\, (T^{1,0}, \,S, \,I), \,(\eta,\, \xi))$ is a pair
$(E, \,\theta)$ consisting of a transversely holomorphic vector bundle $E$ and a section
$\theta\,\in\, A^{1,0}_{B_{\mathcal F_{\xi}}}(M,\,{\rm 
End}(E))$ satisfying the following two conditions:
$$\overline\partial_{E}\theta \,=\,0\ \ \text{ and }
\ \ \theta\wedge \theta\,=\,0\, .$$
This section $\theta$ is called a Higgs field on $E$. Note that
$$\theta\wedge \theta\,\in\, A^{2,0}_{B_{\mathcal F_{\xi}}}(M,\,{\rm
End}(E))\, .$$

For a basic Higgs bundle $(E,\, \theta)$, we define the operator 
$$D^{\prime\prime}_{E,\theta}\,=\,\overline\partial_{E}+\theta\,:\, A^{\ast}_{B_{\mathcal 
F_{\xi}}}(M,\,E)\,\longrightarrow\, A^{\ast+1}_{B_{\mathcal F_{\xi}}}(M,\,E)\, .$$

Let $h$ be a Hermitian metric on $E$. Assume that $h$ is basic.
Define $\overline\theta_{h}\,\in\, A^{0,1}_{B_{\mathcal F_{\xi}}}(M,\,{\rm End}(E))$ by
\begin{equation}\label{tst}
(\theta (e_{1}),\, e_{2})\,=\,(e_{1}, \,\overline\theta_{h} (e_{2}))
\end{equation}
for $e_{1},\,e_{2}\,\in\, E$.

Let $\nabla$ be a unitary connection on $E$ preserving $h$ associated to $h$ such that
$\nabla$ actually restricts as $\nabla\,:\,A^{\ast}_{B_{\mathcal F_{\xi}}}(M,\,E)\,\longrightarrow\,
A^{\ast+1}_{B_{\mathcal F_{\xi}}}(M,\,E)$.
Define the connection $$D\,=\,\nabla+\theta+\overline\theta_{h}\, ,$$ and consider the
curvature $R^{D}\,=\,D^{2}$ of $D$.
Then, by the assumptions on $\theta$ and $h$, we have that $$R^{D}\,\in\,
A^{2}_{B_{\mathcal F_{\xi}}}(M,\,{\rm End}(E))\, .$$
Define the {\it degree} of $E$ to be
 \[{\rm deg}(E)\,=\,\frac{\sqrt{-1}}{2\pi}\int_{M}{\rm Tr} (\Lambda R^{D})\, .
\]
Note that we have
$${\rm deg}(E)\,=\,\int_{M}c_{1, B_{\mathcal F_{\xi}}}(E)\wedge d\eta\wedge\eta\, ,$$ and hence
${\rm deg}(E)$ depends only on $E$.

We define the canonical (Chern) connection $\nabla^{h}$ on the transversely holomorphic Hermitian bundle
$(E,\, h)$ in the following way. Take local basic holomorphic frames $e^{\alpha}_{1},\,\cdots ,\,
e_{\alpha}^{\alpha}$ of $E$ with respect to an open covering $M\,=\,\bigcup_{\alpha} U_{\alpha}$. For the Hermitian 
matrices $H_{\alpha}\,=\,(h_{i\overline{j}}^{\alpha})$ with 
$h_{i\overline{j}}^{\alpha}\,=\,h(e^{\alpha}_{i},e^{\alpha}_{j})$, we define
\begin{equation}\label{nah}
\nabla^{h}\,=\,d+ 
H_{\alpha}^{-1}\partial_{\xi}H_{\alpha}
\end{equation}
on each $U_{\alpha}$.

Let us consider the canonical connection $D^{h}\,=\,\nabla^{h}+\theta+\overline\theta_{h}$
on $E$. We also define the operator 
$\partial_{E,h}\,:\,A^{p,q}_{B_{\mathcal F_{\xi}}}(M,\,E)\,\longrightarrow\,
A^{p+1,q}_{B_{\mathcal F_{\xi}}}(M,\,E)$ 
such that $$\partial_{E,h}\,=\,\partial_{\xi}+ H_{\alpha}^{-1}\partial_{\xi}H_{\alpha}$$ on each 
$U_{\alpha}$; so $\partial_{E,h}$ is the $(1,\, 0)$-component of $\nabla^{h}$.
 
We now define the stable Higgs bundles (cf. \cite{BH}, \cite{BS}). Denote by ${\mathcal 
O}_{B_{\mathcal F_{\xi}}}$ the sheaf of basic holomorphic functions on $M$, and for a 
transversely holomorphic vector bundle $E$ on $M$, denote by ${\mathcal O}_{B_{\mathcal F_{\xi}}}(E)$ 
the sheaf of basic holomorphic sections of $E$. Consider ${\mathcal O}_{B_{\mathcal 
F_{\xi}}}(E)$ as a coherent ${\mathcal O}_{B_{\mathcal F_{\xi}}}$-sheaf.

For a basic Higgs 
bundle $(E,\, \theta)$, a {\em sub-Higgs sheaf} of $(E,\, \theta)$ is a coherent 
${\mathcal O}_{B_{\mathcal F_{\xi}}}$-subsheaf $\mathcal V$ of ${\mathcal O}_{B_{\mathcal 
F_{\xi}}}(E)$ such that $\theta ({\mathcal V})\, \subset\, {\mathcal 
V}\otimes \Omega_{B_{\mathcal F_{\xi}}}$, where $\Omega_{B_{\mathcal F_{\xi}}}$ is the sheaf 
of basic holomorphic $1$-forms on $M$. By \cite[Proposition 3.21]{BH}, if ${\rm rk} 
(\mathcal V)<{\rm rk}(E)$ and ${\mathcal O}_{B_{\mathcal F_{\xi}}}(E)/\mathcal V$ is 
torsion-free, then there is a transversely analytic sub-variety $S\,\subset\, M$ of complex 
co-dimension at least 2 such that ${\mathcal V}$ is given by a transversely holomorphic 
sub-bundle $V\,\subset\, E$; the degree ${\rm deg}(\mathcal V)$ can be defined by integrating on 
this complement $M\setminus S$.

\begin{definition}\label{def1}
We say that a basic Higgs bundle $(E,\, \theta)$ is {\em stable} if $E$ admits a basic Hermitian
metric and for every 
sub-Higgs sheaf ${\mathcal V}$ of $(E,\, \theta)$ such that ${\rm rk} (\mathcal V)\,<\,{\rm 
rk}(E)$ and ${\mathcal O}_{B_{\mathcal F_{\xi}}}(E)/\mathcal V$ is torsion-free, 
the inequality
\[\frac{{\rm deg}(\mathcal V)}{{\rm rk} (\mathcal V)}\,<\,\frac{{\rm deg}(E)}{{\rm rk}(E)}
\]
holds.

A basic Higgs bundle $(E,\, \theta)$ is called {\em polystable} if
$$
(E,\, \theta)\,=\, \bigoplus_{i=1}^k (E_i,\, \theta_i)\, ,
$$
where each $(E_i,\, \theta_i)$ is a stable basic Higgs bundle with
\[\frac{{\rm deg}(\mathcal E_i)}{{\rm rk} (\mathcal E_i)}\,=\,\frac{{\rm deg}(E)}{{\rm rk}(E)}\, .
\]
\end{definition}

\begin{thm}\label{HigHe}
For a stable basic Higgs bundle $(E, \,\theta)$ over a compact Sasakian manifold $(M,\, (T^{1,0},
\, S, \,I),\, (\eta, \,\xi))$, there exists a basic Hermitian metric $h$ on $E$ such that
\[\Lambda R^{D^{h} \perp}\, =\,0\, ,
\]
where $R^{D^{h} \perp}$ is the trace-free part of the curvature $R^{D^{h}}$ of the
canonical connection $D^{h}$ associated to $h$.
\end{thm}
 
Theorem \ref{HigHe} will be proved in the next section.

\begin{prop}\label{HtF}
Let $(E,\, \theta)$ be a basic Higgs bundle over a compact Sasakian manifold
$(M,\, (T^{1,0},\, S,\, I),\, (\eta,\, \xi))$. Suppose that $h$ is a basic Hermitian metric
on $E$ such that $\Lambda R^{D^{h} \perp}\,=\,0$.
If $c_{1, B_{\mathcal F_{\xi}}}(E)\,=\,0$ and $c_{2, B_{\mathcal F_{\xi}}}(E)\,=\,0$, then
the connection $D^{h}$ is flat.
\end{prop}

\begin{proof}
The arguments in the proof of \cite[Proposition 3.4]{SimC} for the usual case
go through after using the Riemann bilinear relations for basic forms on $M$.
\end{proof}

\section{Proof of Theorem \ref{HigHe}}

\subsection{Preliminaries}\label{se6.1}

The proof of Theorem \ref{HigHe} will closely follow the proof of \cite[Theorem 1]{SimC}.

\begin{prop}\label{Ass3}
Let $(M,\, (T^{1,0}, \,S,\, I),\, (\eta,\, \xi))$ be a compact Sasakian manifold. 
Let $B\,\in\, {\mathbb R}^{>0}$ be a positive number.
There exist positive constants $C_{1}(B)$, $C_{2}(B)$ and $C_{3}(B)$ depending on $B$ and an
increasing function $a\,:\, [0,\,+\infty)\,\longrightarrow\, [0,\,+\infty)$ with $a(0)\,=\,0$,
$a(x)\,=\,x$ for $x\,>\,1$, such that 
for any positive basic function $f\,\in\, A^{0}_{B_{\mathcal F_{\xi}}}(M)$ on $M$
satisfying $\Delta_{\xi} f\,\le\, B$, the following two inequalities hold:
 \[\sup_{M} f\,\le\, C_{1}(B)+C_{2}(B)\int_{M}f
 \]
and
\[\ \sup_{M} f \,\le\, C_{3}(B)a\left( \int_{M}f\right)\, .
\]
\end{prop}

\begin{proof}
In view of $\Delta f\,=\,\Delta_{\xi} f$ for $f\,\in\, A^{0}_{B_{\mathcal F_{\xi}}}(M)$, it
suffices to prove the two inequalities for $f\,\in \,A^{0}_{B_{\mathcal F_{\xi}}}(M)$
satisfying $\Delta f\,\le\, B$. We next note that this is already proved in
\cite[Proposition 2.1]{SimC} and \cite{DonI}.
\end{proof}

Let $(E,\, \theta)$ be a basic Higgs bundle, equipped
with a basic Hermitian metric $h$, over a compact Sasakian manifold
$(M,\, (T^{1,0},\, S,\, I),\, (\eta,\, \xi))$.
We define the operator $$D^{\prime, h}_{E,\theta}\,=\,\partial_{E,h}+\overline\theta_{h}\,:\,
A^{\ast}_{B_{\mathcal F_{\xi}}}(M,\,E)\,\longrightarrow\,
A^{\ast+1}_{B_{\mathcal F_{\xi}}}(M,\,E)\, ,$$ where $\partial_{E,h}$ and $\overline\theta_{h}$
are defined as in \eqref{tst} and \eqref{nah} respectively.
As in \cite[Lemma 3.1]{SimC}, we have the following formulas:
\begin{enumerate}
\item \[\sqrt{-1}[\lambda,\,D^{\prime\prime}_{E,\theta}]\,=\,(D^{\prime, h}_{E,\theta})^{\ast}
\qquad{\rm and}\qquad \sqrt{-1}[\lambda,\,D^{\prime,h}_{E,\theta}]\,=\,-(D^{\prime\prime}_{E,\theta})^{\ast}\, ,\]
 where $(D^{\prime, h}_{E,\theta})^{\ast}\,=\,-\star_{h,\xi} \overline\partial_{E}
\star_{h,\xi}+\star_{h,\xi} \theta \star_{h,\xi}$ and $$(D^{\prime\prime}_{E,\theta})^{\ast}
\,=\,-\star_{h,\xi}\partial_{E,h}\star_{h,\xi} +
\star_{h,\xi} \overline\theta_{\theta} \star_{h,\xi}\, ;$$ note that they are the formal adjoints
of $D^{\prime, h}_{E,\theta}$ and $D^{\prime\prime}_{E,\theta}$ respectively for the $L^{2}$ inner product
$A^{r}_{B_{\mathcal F_{\xi}}}(M,\, E)\times A^{r}_{B_{\mathcal F_{\xi}}}(M,\,E)\,\ni\,
(\alpha,\beta)\,\longmapsto \,\int_{M} \langle \alpha,\, \beta\rangle$.

 \item For self-adjoint basic sections $\sigma, \tau\,\in \,
A^{0}_{B_{\mathcal F_{\xi}}}(M,\,{\rm end}(E))$,
 \[\vert D^{\prime\prime}_{E,\theta}(\sigma)\tau\vert_{h}^{2}\,=\,
-\sqrt{-1}\Lambda {\rm Tr}( D^{\prime\prime}_{E,\theta}(\sigma)\tau^{2}D^{\prime, h}_{E,
\theta}(\sigma)).
 \]

 \item If ${\mathbf k}\,=\,h\sigma$ for a basic positive self-adjoint section $\sigma\,\in\,
A^{0}_{B_{\mathcal F_{\xi}}}(M,\,{\rm end}(E))$, then
 \[D^{\prime, {\mathbf k}}_{E,\theta}\,=\,D^{\prime, h}_{E,\theta}+\sigma^{-1}D^{\prime, h}_{E,\theta}(\sigma)\]
 and
 \[\Delta_{h}^{\prime}(\sigma)\,=\,h\sqrt{-1}(\Lambda R^{D^{\mathbf k}}-\Lambda R^{D^{h}})+\sqrt{-1}\Lambda D^{\prime\prime}_{E,\theta}(\sigma)\sigma^{-1}D^{\prime, h}_{E,\theta}(\sigma)\, ,
 \]
 where $\Delta_{h}^{\prime}\,=\,(D^{\prime, h}_{E,\theta})^{\ast}D^{\prime, h}_{E,\theta}
\,=\,\sqrt{-1}\Lambda D^{\prime\prime}_{E,\theta}D^{\prime, h}_{E,\theta}$.
 
\item Also,
\begin{equation}\label{eqinq}
\Delta_{\xi} \log {\rm Tr} (\sigma)\,\le\,
2(\vert\Lambda R^{D^{\mathbf k}}\vert_{\mathbf k}+\vert\Lambda R^{D^{h}}\vert_{h})\, .
\end{equation}
\end{enumerate}

\subsection{Donaldson's functional}

Let $(E,\, \theta)$ be a basic Higgs vector bundle over the
Sasakian manifold $(M,\, (T^{1,0}, \,S,\, I),\, (\eta, 
\,\xi))$. Fix a basic Hermitian metric $h$ on $E$. Denote by $S(E)$ the smooth vector 
bundle of self-adjoint endomorphisms of $E$ corresponding to $h$. For smooth functions 
$\phi\,:\, \R\,\longrightarrow\, \R$ and $\Psi\,:\, \R\times \R\,\longrightarrow\, \R$, we
define the maps $\phi\,:\, S(E)\,\longrightarrow\, S(E)$ and 
$\Psi\,:\, S(E)\,\longrightarrow\, S({\rm End} E)$ as in \cite[Page 880]{SimC}. Consider
the $L^{p}$-completion 
$L_{B_{\mathcal F_{\xi}}}^{p}(S(E))$ of the space of basic sections of $S(E)$, and also the 
Sobolev space $L^{p,1}_{B_{\mathcal F_{\xi}}}(S(E))$ with the norm $\Arrowvert s 
\Arrowvert_{L^{p,1}_{B_{\mathcal F_{\xi}}}(S(E))} \,=\,\Arrowvert s 
\Arrowvert_{L^{p}}+\Arrowvert D^{\prime\prime}_{E,\theta}s \Arrowvert_{L^{p}}$ for any
section $s$ of $S(E)$. As done in \cite[Proposition 4.1]{SimC}, we can extend $\phi\,:\, \R
\,\longrightarrow\, \R$ and 
$\Psi\,:\, \R\times \R\,\longrightarrow\, \R$ to continuous maps $$\phi\,:\,
L_{B_{\mathcal F_{\xi}}}^{p}(S(E))\,\longrightarrow\, L_{B_{\mathcal F_{\xi}}}^{p}(S(E))\, ,$$
$$\Psi\,:\, L_{B_{\mathcal F_{\xi}}}^{p}(S(E))\,\longrightarrow\, {\rm 
Hom}(L^{p}_{B_{\mathcal F_{\xi}}}({\rm End}(E),\, L^{q}_{B_{\mathcal F_{\xi}}}({\rm End}(E))$$
and $\phi\,:\,L^{p,1}_{B_{\mathcal F_{\xi}}}(S(E))\,\longrightarrow\, L^{q,1}_{B_{\mathcal F_{\xi}}}(S(E))$ for 
$q\,<\,p$.
 
Denote by $\mathcal P$ the space of all basic Hermitian metrics on $E$.
We parametrize $\mathcal P$ by $A^{0}_{B_{\mathcal F_{\xi}}}(S(E))$ as
$A^{0}_{B_{\mathcal F_{\xi}}}(S(E))\,\ni\, \sigma\,\longmapsto \,h\exp ({\sigma})\,\in\, \mathcal P$.
For $h,\,h^{\prime}\,\in\, \mathcal P$ with $h^{\prime}\,=\, h\exp (s)$, define
\[M(h,\,h^{\prime})\,=\,\sqrt{-1}\int_{M} {\rm Tr}(s \Lambda R^{D^{h}})+
\int_{M}\left\langle \Psi_1 (s)(D^{\prime\prime}_{E,\theta}s), \,D^{\prime\prime}_{E,\theta}s\right\rangle_{h}\, ,
\]
where $\Psi_1\,:\, S(E)\,\longrightarrow\, S({\rm End} E)$ is defined by the function 
\[\Psi_1(\lambda_{1},\lambda_{2})\,=\,
\frac{e^{\lambda_{2}-\lambda_{1}}-(\lambda_{2}-\lambda_{1})-1}{(\lambda_{2}-\lambda_{1})^{2}}\, .
\]
By the same proof as of \cite[Proposition 5.1]{SimC}, for all $h,\,h^{\prime},\,h^{\prime\prime}
\,\in\, \mathcal P$, the identity
\[M(h,h^{\prime})+M(h^{\prime},h^{\prime\prime})\,=\,M(h,h^{\prime\prime})
\]
is obtained.

We have the following important estimate as in \cite[Proposition 5.3]{SimC}.

\begin{prop}\label{ESTm}
Fix a positive number $B$.
Suppose a basic Hermitian metric $h$ on $E$ satisfies the inequality
$\sup_{M}\vert \Lambda R^{D^{h}}\vert \,\le\, B$.
If $(E,\, \theta)$ is a stable
basic Higgs bundle, then there are positive constants $C_{1}$ and $C_{2}$ such that 
\[\sup_{M}\vert s\vert\,\le\, C_{1}+C_{2}M(h, h\exp (s))
\]
for every $s\,\in\, A^{0}_{B}(S(E))$ with ${\rm Tr}(s)\,=\,0$ and
$\sup_{M}\vert \Lambda R^{D^{he^{s}}}\vert \,\le\, B$.
\end{prop}

Proposition \ref{ESTm} will be proved after Theorem \ref{TrUY}.

\begin{definition}
Let $E$ be a transversely holomorphic vector bundle, over a compact Sasakian manifold $M$,
equipped with a basic Hermitian metric $h$.
 A {\em transversely weakly holomorphic subbundle} of $E$ is $\Pi\,\in\, L^{2,1}_{B}(S(E))$ satisfying
the following two conditions
\[\Pi\,=\,\Pi^{2}\qquad {\rm and } \qquad ({\rm Id}_{E}-\Pi)\overline\partial_{{\rm End}(E)}
(\Pi)\,=\,0\, .
\] 
\end{definition}
 
In \cite[Theorem 5.7]{BH}, the following result is proved.
 
\begin{thm}[\cite{BH}]\label{TrUY}
Let $\Pi$ be a transversely weakly holomorphic subbundle of $E$.
Then there is a transversely coherent sheaf $\mathcal V$ and a transverse analytic subset
$S\,\subset\, M$ such that the following three hold:
\begin{enumerate}
\item The complex codimension of $S$ in $M$ is at least two.

\item The restriction of $\Pi$ to $M\setminus S$ is smooth and defines a transversely holomorphic 
subbundle $V\,\subset\, E\vert_{M\setminus S}$.
 
\item The restriction of $\mathcal V$ to $M\setminus S$ is the sheaf of basic holomorphic sections of $V$.
Moreover these properties imply that ${\mathcal O}_{B_{\mathcal F_{\xi}}}(E)/\mathcal V$ is
torsion-free.
\end{enumerate}
\end{thm}
 
Now we can prove Proposition \ref{ESTm} in the same way as done for the proof of \cite[Proposition 
5.3]{SimC}.

\begin{proof}[{Proof of Proposition \ref{ESTm}}]
For a section $s$ as in the statement of the proposition,
by the inequality $$\Delta_{\xi} \log {\rm Tr} (\sigma)\,
\le\, 2(\vert\Lambda R^{D^{\mathbf k}}\vert_{\mathbf k}+\vert\Lambda R^{D^{h}}\vert_{h})$$
in \eqref{eqinq} we have
that $\Delta_{\xi} \vert s\vert \,\le\, 4B$, and hence using Proposition \ref{Ass3} it
follows that $\sup_{M}\vert s\vert\,\le\, C_{1}+C_{2}\Arrowvert s\Arrowvert _{L^1}$.

Assume that the estimate in the statement of the proposition does not hold
in the following sense.
We can find a sequence $s_{i}$ of basic sections of $S(E)$ with ${\rm Tr}(s_{i})\,=\,0$ such that 
$\lim_{i\to \infty} \Arrowvert s_{i}\Arrowvert _{L^1}\,=\,+\infty$ and
$\Arrowvert s_{i}\Arrowvert _{L^1}\,\ge\, M(h,\, h\exp (s_{i}))$.

Then, by the same arguments given between \cite[Lemma 5.4]{SimC} and \cite[Lemma 5.7]{SimC}, we have
a transversely weakly holomorphic subbundle $\Pi$ of $E$ such that $({\rm Id}_{E}-\Pi)\theta(\Pi)
\,=\,0$, and 
the inequality 
\[\frac{{\rm deg}(\Pi)}{{\rm Tr}(\Pi)}\,\ge\, \frac{{\rm deg}(E)}{{\rm rk}(E)}
\]
holds, where 
\[{\rm deg}(\Pi)\,=\,
\sqrt{-1}\int_{M}{\rm Tr}(\Pi\Lambda R^{D^{h}})-\int_{M}\vert D^{\prime\prime}_{E,\theta}\Pi\vert^{2}\, .
\]
We have a sub-Higgs sheaf $\mathcal V$ of $(E, \,\theta)$ satisfying the properties in
Theorem \ref{TrUY}.
We conclude that ${\rm deg}(\Pi)\,=\,{\rm deg}(\mathcal V)$ as in \cite[Lemma 3.2]{SimC},
and this contradicts the given stability condition. This completes the
proof of Proposition \ref{ESTm}
\end{proof}
 
\subsection{The heat equation on K\"ahler cone}
 
For a compact Sasakian manifold $$(M,\, (T^{1,0},\, S,\, I), \,(\eta,\, \xi))\, ,$$ consider
the cone $C(M)\,=\,M\times \R^{>0}$
and define the real $2$-form $\omega\,\in\, A^{2}(C(M))$ to be
\[\omega\,=\,2rdr\wedge\eta+r^{2}d\eta\, ;
\]
also, define the bundle homomorphism $J\,:\,TC(M)\,\longrightarrow\, TC(M)$ by
\begin{itemize}
\item $J(X)\,=\,I(X)$ for $X\,\in\, S$,

\item $J\left(r\frac{\partial}{\partial r}\right)\,=\,-\xi$, and

\item $J(\xi) \,=\,r\frac{\partial}{\partial r}$, where $r$ is the parameter of $\R^{>0}$.
\end{itemize}
Then, the pair $(\omega,\, J)$ is a K\"ahler structure on the complex manifold $C(M)$.

Consider the real $1$-dimensional foliation $\overline{\mathcal F}_{\xi}$ on $C(M)$ generated by 
$\xi$. Denote by $A$ the $1$-parameter group of automorphisms of the K\"ahler manifold 
$C(M)$ corresponding to $\overline{\mathcal F}_{\xi}$.
Then the space $$A^{\ast}(C(M))^{A}\,\subset\,A^{\ast}(C(M))$$ of $A$-invariant differential forms on $C(M)$ 
contains the basic de Rham complex $A^{\ast}_{B_{\overline{\mathcal F}_{\xi}}}(C(M))$. In 
particular, we have $A^{0}(C(M))^{A}\,=\,A^{0}_{B_{\overline{\mathcal F}_{\xi}}}(C(M))$. For a 
complex basic vector bundle $E$ over $C(M)$, we can naturally define the $A$-action on 
$A^{\ast}(C(M),\, E)$ so that for $a\,\in\, A$, $\omega\,\in\, A^{r}(C(M))$ and $s
\,\in\, A^{0}_{B_{\overline{\mathcal F}_{\xi}}}(C(M),\, E)$, 
we have $a(\omega \otimes s)\,=\,(a^{\ast}\omega)\otimes s$.
 
Let $(E,\, \theta)$ be a basic Higgs bundle over $(M,\, (T^{1,0}, \,S,\, I),\, (\eta,\, 
\xi))$. Consider the pair $(\widetilde{E},\, \widetilde\theta)$ defined by the pull back of 
$(E,\,\theta)$ using the projection
\begin{equation}\label{pi0}
\pi_0\, :\, C(M)\,\longrightarrow\, M
\end{equation}
defined by $(x,\, t)\, \longmapsto\, x$. Then $(\widetilde{E},\,
\widetilde\theta)$ is a 
usual Higgs bundle over the K\"ahler manifold $(C(M),\,\omega, \,J)$. Consider the operator 
$$\widetilde{D}^{\prime\prime}\,=\,\overline\partial_{\widetilde{E}}+\theta 
\,:\,A^{\ast}(C(M),\,\widetilde{E})\,\longrightarrow\, A^{\ast+1}(C(M),\,\widetilde{E})\, ,$$ where 
$\overline\partial_{\widetilde{E}}$ is the usual Dolbeault operator for the holomorphic 
vector bundle $\widetilde{E}$. Let $\widetilde{h}$ be a Hermitian metric on 
$\widetilde{E}$. Assume that $\widetilde{h}$ is basic, i.e., $\widetilde{h}\,\in\, 
A^{0}_{B_{\overline{\mathcal F}_{\xi}}}(C(M),\,\widetilde{E}^{\ast}\otimes 
\overline{\widetilde{E}}^{\ast})$. Then, we have the operator $$\partial_{E, \widetilde{h}}\,:\, 
A^{p,q}(C(M),\,\widetilde{E})\,\longrightarrow\, A^{p,q+1}(C(M),\,\widetilde{E})$$ such that 
$\overline\partial_{\widetilde{E}}+\partial_{E, \widetilde{h}}$ is the canonical
Chern connection 
for the Hermitian holomorphic vector bundle $(\widetilde{E}, \,\widetilde{h})$. Define 
$\overline{\widetilde\theta}_{\widetilde{h}}\,\in\, A^{0,1}(C(M),\,{\rm End}(\widetilde{E}))$ by 
$$h(\widetilde\theta e_{1},\,e_{2})\,=\,\widetilde{h}(e_{1},\, 
\overline{\widetilde\theta}_{\widetilde{h}}e_{2})$$ and also define the operator 
$\widetilde{D}^{\prime}_{\widetilde{h}}\,=\,\partial_{E, 
\widetilde{h}}+\overline{\widetilde\theta}_{\widetilde{h}}\,:\,A^{\ast}(C(M),\,\widetilde{E})
\,\longrightarrow\, 
A^{\ast+1}(C(M),\,\widetilde{E})$. The curvature of 
the connection $\widetilde{D}^{\prime\prime}+\widetilde{D}^{\prime}_{\widetilde{h}}$ will
be denoted by $\widetilde{R}^{\widetilde{h}}$.
 
Denote by $\widetilde\Lambda$ the formal adjoint of the Lefschetz operator associated with the
K\"ahler form $\omega$ on $C(M)$. Consider the heat equation 
\begin{equation}\label{HE}
\widetilde{h}^{-1}_{t}\frac{d\widetilde{h}_{t
}}{dt}\,=\,-\sqrt{-1}\widetilde\Lambda \widetilde{R}^{\widetilde{h}_{t} \perp}\, ,
\end{equation}
where $\widetilde{R}^{\widetilde{h}_{t} \perp}$ is the trace-free part of
the curvature $\widetilde{R}^{\widetilde{h}_{t} }$.

Fix a basic Hermitian metric $h_{0}$ on the basic holomorphic bundle $E$ over the Sasakian
manifold $M$. Let $\widetilde{h}_{0}\, =\, \pi^*_0 h_0$ be the pull-back of $h_{0}$
to a Hermitian metric on $\widetilde{E}\, :=\, \pi^*_0 E$, where $\pi_0$ is the
projection in \eqref{pi0}.

Write $\widetilde{h}_{0}^{-1}\widetilde{h}_{t}\,=\,\sigma_{t}$.
Then the equation in \eqref{HE} becomes 
 \[\left(\frac{d}{dt}+\Delta^{\prime}_{\widetilde{h}_{0}}\right)\sigma_{t}\,=\,
-\sqrt{-1}\sigma_{t}\widetilde\Lambda \widetilde{R}^{\widetilde{h}_{0} \perp}+
\sqrt{-1}\widetilde\Lambda \widetilde{D}^{\prime\prime}(\sigma_{t})\sigma^{-1}_{t}
\widetilde{D}^{\prime}_{\widetilde{h}_{0}}(\sigma_{t})\, ,
 \]
where $\Delta^{\prime}_{h_{0}}$ is the Laplacian operator of $\widetilde{D}^{\prime}_{\widetilde{h_{0}}}$.
Recall that $A$ is a group of automorphisms of the K\"ahler manifold $C(M)$. Since the
Hermitian metric $h_{0}$ on $E$ is basic,
the action of $A$ on $A^{0}(C(M),\,{\rm End}(\widetilde{E}))$ commutes with the operators
$\widetilde{D}^{\prime\prime}$, $\widetilde{D}^{\prime}_{h_{0}}$, $\widetilde\Lambda$
and $\Delta^{\prime}_{h_{0}}$, and we have that $$\widetilde\Lambda\widetilde{R}^{\widetilde{h_{0}} }
\,\in\, A^{0}_{B_{\overline{\mathcal F}_{\xi}}}(C(M),\,{\rm End}(\widetilde{E}))\, .$$
Thus, the set of solutions of the heat equation \eqref{HE} is invariant under the action of $A$.

The positive function $ r^{2}$ on $C(M)$ is plurisubharmonic, because
$\sqrt{-1}\partial\overline\partial r^{2}\,=\,\omega$.
Just as the results in \cite[Section 6]{SimC} are derived using the arguments in
\cite{Ham} and \cite{Don}, we have the following.

\begin{thm}\label{Ellp}
Let $\epsilon$ be a positive real number.
There exists a unique solution $\widetilde{h}$, defined for all time $(0,\,+\infty)$, of the heat equation 
 \[\widetilde{h}_{t}^{-1}\frac{d\widetilde{h}_{t}}{dt}\,=\,-\sqrt{-1}\widetilde\Lambda \widetilde{R}^{\widetilde{h}_{t} \perp} \]
on the compact manifold $M\times [1,\,1+\epsilon]$ with boundary satisfying
$${\rm det}(\widetilde{h}_{0})\,=\,{\rm det}(\widetilde{h}_{t})\, , \ \ \widetilde{h_{t}}_{\vert t=0}
\,=\,\widetilde{h}_{0}$$ together with the Neumann boundary 
condition $\frac{\partial h}{\partial r} \big{\vert}_{ r=1, 1+\epsilon}\,=\,0$.
\end{thm}

\subsection{Proof of Theorem \ref{HigHe}}

Consider the solution $\widetilde{h}_{t}$ in Theorem \ref{Ellp}.
It was observed above that the set of solutions of the heat equation is
$A$-invariant. Therefore, we conclude that $\widetilde {h_{t}}$ is basic by the uniqueness
property. Define the Hermitian metrics $h_{t}$ on $E$ by the pull-backs of $\widetilde {h_{t}}$
for the embedding $M\,\longrightarrow\, M\times [1,\,1+\epsilon]$
defined by $x\, \longmapsto\, (x,\, 1)$. In view of the Neumann boundary 
condition $\frac{\partial h}{\partial r} \big{\vert}_{ r=1, 1+\epsilon}\,=\,0$
in Theorem \ref{Ellp}, we conclude that
the pull-back, by this embedding, of the canonical Chern connection $\overline\partial_{\widetilde{E}}+
\partial_{E, \widetilde{h}_{t}}$, on $(\widetilde{E},\, \widetilde{h}_{t})$, is identified with the
canonical Chern connection $\nabla^{h_{t}}$ on $E$, and moreover the pull-back of 
$(\widetilde\Lambda \widetilde{R}^{\widetilde{h} \perp})$ is identified
with $\Lambda R^{D^{h_{t}} \perp}$.
Thus, $h_{t}$ satisfies the basic heat equation 
 \[h_{t}^{-1}\frac{d h_{t}}{dt}\,=\,-\sqrt{-1}\Lambda R^{D^{h_{t}} \perp}\, . \]

By the formulas in Section \ref{se6.1}, for $h_{0}^{-1}h_{t}\,=\,\sigma_{t}$,
this equation can be written as
\[\left(\frac{d}{dt}+\Delta^{\prime}_{{h}_{0}}\right)\sigma_{t}\,=\,
-\sqrt{-1}\sigma_{t}\Lambda {R}^{{h}_{0} \perp}+\sqrt{-1}\Lambda D^{\prime\prime}(
\sigma_{t})\sigma^{-1}_{t}D^{\prime}_{h_{0}}(\sigma_{t})\, .
\]
 
By the same proof of \cite[Lemma 7.1]{SimC}, we have the formula
\[\frac{d}{dt}M(h_{0}, h_{t})\,=\,-\int_{M} \vert \Lambda R^{D^{h_{t}} \perp}\vert^{2}_{h_{t}}\, .
\] 

Now we assume that the Higgs bundle $(E,\, \theta)$ is stable.
Applying Proposition \ref{ESTm}, as done in \cite[Page 895]{SimC}, we can take a sequence
$\{t_{i}\} $ of time instances, with $t_{i}\,\to\, +\infty$, such that
$$\lim_{i\to +\infty}\int_{M} \vert \Lambda R^{D^{h_{t}} \perp}\vert^{2}_{h_{t_{i}}}\,=\,0$$
and $h_{t_{i}}\,\to\, h_{\infty}$ weakly in $L^{2}_{1}$.
By the basic Sobolev embedding theorem \cite[Theorem 2.6]{BH}, this $h_{t_{i}}$ is a Cauchy
sequence in $L^{1}$.
For a positive number $B$ such that $\vert \Lambda R^{D^{h_{0}}}\vert \,\le\, B$, by the
inequality $\Delta_{\xi} \log {\rm Tr} (\sigma)\,\le\, 2(\vert\Lambda R^{D^{k}}\vert_{k}+
\vert\Lambda R^{D^{h}}\vert_{h})$ in \eqref{eqinq}, we have that
 \[\Delta_{\xi} \log {\rm Tr}(h_{t_{i}}^{-1}h_{t_{j}})\,\le\, 2B\, ,
 \]
and hence Proposition \ref{Ass3} implies that $\log {\rm Tr}(h_{t_{i}}h_{t_{j}})\,\to\, 0$ in $C^{0}$.
Thus the convergence $h_{t_{i}}^{-1}\,\to\, h_{\infty}$ is in $C^{0}$.
By the $C^{0}$-convergence and the uniform boundedness of
$\vert\Lambda R^{D^{h_{t}}\perp}\vert_{h_{t}}$, as done in the arguments in the proof
of \cite[Lemma 19]{Don} (also \cite[Lemma 6.4]{SimC}), we
conclude that $h_{t_{i}}$ is actually bounded in $C^{1}$.
 
To complete the proof, we use the basic elliptic estimate and regularity explained in \cite[Section 2]{BH}.
For the transversely elliptic operator $\Delta_{h_{0}}^{\prime}$, as in \cite[Remark 2.2]{BH}, we
can non-canonically extend $\Delta_{h_{0}}^{\prime}$ to a second order differential operator
$$L\,:\, A^{0}(S(E))\,\longrightarrow\, A^{0}({\rm End}(E))\, .$$
For the linear operator $$\nabla^{h}_{\xi}\,:\,A^{0}(S(E))\,\longrightarrow\, A^{0}(S(E))$$ associated with the connection $\nabla^{h}$, and its formal adjoint $(\nabla^{h}_{\xi})^{\ast}$,
define the second order differential operator $$\Box\,:=\,
(\nabla^{h}_{\xi})^{\ast}\nabla^{h}_{\xi}+L\,:\, A^{0}(S(E))
\,\longrightarrow\, A^{0}({\rm End}(E))\, .$$
Then, by the transverse ellipticity of $\Delta_{h_{0}}^{\prime}$, the differential operator 
$\Box$ is elliptic. We have $\Box=\Delta_{h_{0}}^{\prime}$ on $A^{0}_{B_{\mathcal 
F_{\xi}}}(S(E))$. Therefore, by the
\begin{itemize}
\item elliptic estimate of the elliptic operator $\Box$,

\item the $C^{1}$-boundedness of $h_{t_{i}}$, and

\item the uniform boundedness of $\vert\Lambda 
R^{D^{h_{t}}\perp}\vert_{h_{t}}$,
\end{itemize}
we conclude that $h_{t_{i}}$ is bounded in $L^{p,2}$, and 
hence the convergence $h_{t_{i}}\to h_{\infty}$ is weakly in $L^{p,2}$. Thus, 
$R^{D^{h_{\infty}}}$ is defined, and $\Lambda R^{D^{h_{\infty}}\perp}\,=\,0$.

We shall prove that $h_{\infty}$ is a smooth basic Hermitian metric.
For that it is sufficient to show that $\sigma_{\infty}\,=\,
h_{0}^{-1}h_{\infty} \in A^{0}_{B_{\mathcal F_{\xi}}}(S(E))$.
We consider the Sobolev space $L^{p,k}(S(E))$.
Then, the basic Sobolev space $L^{p,1}_{B_{\mathcal F_{\xi}}}(S(E))$ is the closure of
$A^{0}_{B_{\mathcal F_{\xi}}}(S(E))$ in $L^{p,1}(S(E))$.
By the elliptic regularity for the elliptic operator $\Box$ (see \cite[Lemma 2.8]{BH}), we
conclude
that $$\sigma_{\infty} \,\in\, L^{p,1}_{B_{\mathcal F_{\xi}}}(S(E))\cap A^{0}(S(E))\, .$$

We have $A^{0}_{B_{\mathcal F_{\xi}}}(S(E))\,=\,{\rm kernel} (\nabla^{h})$ for the linear
differential operator
$\nabla^{h}_{\xi}\,:\,A^{0}(S(E))\,\longrightarrow\, A^{0}(S(E))$.
Extend $(\nabla^{h}_{\xi})^{p,1}\,:\,L^{p,1}(S(E))\,\longrightarrow\, L^{p,0}(S(E))$ continuously. 
Then we have $L^{p,1}_{B_{\mathcal F_{\xi}}}(S(E))\,\subset\, {\rm kernel}((\nabla^{h})^{p,1})$.
From the commutative diagram
 \[\xymatrix{A^{0}(S(E))\ar[r]\ar[d]^{\nabla^{h}_{\xi}}&L^{p,1}(S(E)) \ar[d]^{(\nabla^{h}_{\xi})^{p,1}}\\
A^{0}(S(E)) \ar[r]&L^{p,0}(S(E))
}
\]
where $A^{0}(S(E))\,\longrightarrow\, L^{p,1}(S(E))$ and $A^{0}(S(E))
\,\longrightarrow\, L^{p,0}(S(E))$ are the natural
inclusions, it follows that ${\rm kernel}((\nabla^{h}_{\xi})^{p,1})\cap A^{0}(S(E))\,\subset\,
{\rm kernel}(\nabla^{h}_{\xi})$. Thus we have $\sigma_{\infty} \,\in \, A^{0}_{B_{\mathcal F_{\xi}}}(S(E))$.
This completes the proof of Theorem \ref{HigHe}.

\subsection{Bogomolov--Miyaoka inequality}

The following Bogomolov--Miyaoka type inequality is derived from Theorem \ref{HigHe} just as
\cite[p.~878--879, Proposition 3.4]{SimC} is derived from \cite[p.~878, Theorem 1]{SimC}.

\begin{cor}\label{cor1}
Let $(E, \,\theta)$ be a polystable basic Higgs bundle of rank $r$ over a compact Sasakian manifold $(M,\, (T^{1,0},
\, S, \,I),\, (\eta, \,\xi))$ with $\dim M\,=\, 2n+1$. Then
$$
\int_{M} \left(2c_{2, B_{\mathcal F_{\xi}}}(E) -
\frac{r-1}{r}c_{1, B_{\mathcal F_{\xi}}}(E)^2\right)(d\eta)^{n-2}\wedge\eta \, \geq\, 0\, ,
$$
where $c_{i, B_{\mathcal F_{\xi}}}(E)$ is the $i$-th basic Chern class of $E$.
If the above inequality is an equality,, then $R^{D^{h} \perp}\, =\,0$.
\end{cor}

\section{Correspondence between flat bundles and Higgs bundles}

\begin{prop}[{See \cite[Theorem 4.7]{BH}}]\label{Porpf}
Let $(E,\, \theta)$ be a basic Higgs bundle over a compact Sasakian manifold
$(M,\, (T^{1,0},\, S,\, I),\, (\eta,\, \xi))$ with ${\rm deg}(E)=0$. 
Suppose that $h$ is a basic Hermitian metric
on $E$ with $\Lambda R^{D^{h} }\,=\,0$.
Then $(E,\, \theta)$ is a direct sum of stable basic Higgs bundles of degree zero.
\end{prop}

\begin{proof}
Assume that $(E,\, \theta)$ is not stable.
Then there exists a 
sub-Higgs sheaf ${\mathcal V}$ of $(E,\, \theta)$ such that
\begin{itemize}
\item ${\rm rk} (\mathcal V)\,<\,{\rm 
rk}(E)$ with ${\mathcal O}_{B_{\mathcal F_{\xi}}}(E)/\mathcal V$ is torsion-free, and 

\item the inequality
\[{\rm deg}(\mathcal V)\,\ge\, 0
\]
holds.
\end{itemize}
Let $\pi$ be the projection to the transversely holomorphic subbundle $V\,\subset\, E$ defined 
almost everywhere using $\mathcal V$, constructed as in \cite[Proposition 3.21]{BH} associated
with the Hermitian 
metric $h$. Then, $\pi\in L^{2,1}_{B}(S(E))$, and we have the Chern-Weil formula (\cite[Lemma 
3.2]{SimC})
\[{\rm deg}(\mathcal V)\,=\,
\sqrt{-1}\int_{M}{\rm Tr}(\pi\Lambda R^{D^{h}})-\int_{M}\vert D^{\prime\prime}_{E,\theta}\pi\vert^{2}\, .
\]
By $\Lambda R^{D^{h} }\,=\,0$, we have that ${\rm deg}(\mathcal V)\,=\,0$
and $D^{\prime\prime}_{E,\theta}\pi\,=\,0$.
By $\pi\,\in\, L^{2,1}_{B}(S(E))$, we also have that $D^{\prime,h}_{E,\theta}\pi\,=\,0$.
Using the elliptic regularity it follows that
$\pi\,\in\, L^{p,1}_{B_{\mathcal F_{\xi}}}(S(E))\cap A^{0}(S(E))$, and hence
$\pi\,\in \, A^{0}_{B_{\mathcal F_{\xi}}}(S(E))$ by the same argument as in the last part of
the proof of Theorem \ref{HigHe}.

Using $D^{\prime\prime}_{E,\theta}\pi\,=\,0$ and $D^{\prime,h}_{E,\theta}\pi\,=\,0$, it can be 
seen that $\pi$ is the projection to a globally defined basic Higgs sub-bundle $V\,\subset\, 
E$ of $(E,\, \theta)$. Thus have the direct sum decomposition
$E\,=\,V\oplus V^{\perp}$ of basic Higgs bundles that satisfies the condition that
${\rm deg}( V)\,=\,0\,=\, {\rm deg}( V^{\perp})$. Restricting $h$ to $V$ and $V^{\perp}$,
and repeating the arguments the proposition can now be proved inductively.
\end{proof}
 
In view of Theorem \ref{thm1}, Theorem \ref{HigHe}, Proposition \ref{Porpf} and Proposition 
\ref{HtF}, on a compact Sasakian manifold $(M,\, (T^{1,0},\, S,\, I),\, (\eta,\, \xi))$, we 
have a bijective correspondence between
\begin{itemize}
\item the semi-simple flat bundles $(E,\,\nabla_{E})$, and

\item the polystable basic Higgs bundles $(E,\,\theta)$ over $M$
with $c_{1, B_{\mathcal F_{\xi}}}(E)\,=\,0\,=\, c_{2, B_{\mathcal F_{\xi}}}(E)$.
\end{itemize}
via harmonic metrics $h$. It should be clarified that the transversely holomorphic structure
of the vector bundle underlying $(E,\,\nabla_{E})$ is in general different from the
transversely holomorphic structure of the vector bundle underlying the corresponding
basic Higgs bundle $(E,\,\theta)$; however the underlying $C^\infty$ vector bundles
coincide.

In the above correspondence, considering the basic de Rham complex $(A^{\ast}_{B_{\mathcal 
F_{\xi}}}(M,E), \,d_{E})$ with values in a semi-simple flat bundle $(E,\,\nabla^{E})$, we have the 
decomposition $$d_{E}\,=\,D^{\prime\prime}_{E,\theta}+D^{\prime,h}_{E,\theta}\, .$$ Denote by 
$$H^{\ast}_{dR, B_{\mathcal F_{\xi}}}(M,\,E)\ \ \text{ and }\ \
H^{\ast}_{Dol, B_{\mathcal F_{\xi}}}(M,\,E)$$ the 
cohomologies of complexes $(A^{\ast}_{B_{\mathcal F_{\xi}}}(M,E),\, d_{E})$ and 
$(A^{\ast}_{B_{\mathcal F_{\xi}}}(M,\,E),\, D^{\prime\prime}_{E,\theta})$ respectively. By the 
K\"ahler identities, we have an isomorphism
\[H^{\ast}_{dR, B_{\mathcal F_{\xi}}}(M,\,E)\,\cong\, H^{\ast}_{Dol, B_{\mathcal F_{\xi}}}(M,\,E)
\]
by the transverse Hodge theory of usual basic cohomology (see \cite{EKA}).
In particular, a section $\varphi\,\in\, A^{\ast}_{B_{\mathcal F_{\xi}}}(M,\,E)$ is flat (i.e.,
$\nabla^{E}\varphi \,=\,0$) if and only if $D^{\prime\prime}_{E,\theta}\varphi\,=\,0$.
 
As done in \cite[Corollary 1.3]{SimL}, we now obtain the following result.

\begin{thm}\label{thml}
Let $(M,\, (T^{1,0},\, S,\, I),\, (\eta,\, \xi))$ a compact Sasakian manifold. Then
there is an equivalence between the category of semi-simple flat vector bundles over $M$ and the 
category of polystable basic Higgs bundles over $M$
with trivial first and second basic Chern classes.
\end{thm}

\section*{Acknowledgements}

We thank T. Mochizuki for pointing out Corollary \ref{cor1}. We thank
D. Baraglia for his comments. The first-named author is
partially supported by a J. C. Bose Fellowship.
The second-named author is was partially supported by JSPS Grant-in-Aid for Scientific Research Project/Area Number	19H01787.

\end{document}